\documentclass[12pt]{amsart}
\usepackage{amsmath,latexsym,amscd,amsbsy,amssymb,amsfonts,amsthm,fleqn,leqno,
euscript, graphicx, texdraw, pb-diagram}
\usepackage[matrix,arrow,curve]{xy}

\numberwithin{equation}{section}
\newtheorem{thm}{Theorem}[section]
\newtheorem{pro}[thm]{Proposition}
\newtheorem{lem}[thm]{Lemma}
\newtheorem{cor}[thm]{Corollary}

\theoremstyle{definition}

\theoremstyle{remark}

\newcommand{\Hom}{\mathrm{Hom}}

\newcommand{\bd}{\mathrm{bd}}

\numberwithin{equation}{section}

\begin{document}

\title[Homological dimension and dimensional full-valuedness]
{Homological dimension and dimensional full-valuedness}

\author{V. Valov}
\address{Department of Computer Science and Mathematics,
Nipissing University, 100 College Drive, P.O. Box 5002, North Bay,
ON, P1B 8L7, Canada} \email{veskov@nipissingu.ca}

\thanks{The author was partially supported by NSERC
Grant 261914-13.}


\begin{abstract}
There are different definitions of homological dimension of metric compacta involving either \v{C}ech homology or exact (Steenrod) homology. In this paper we investigate the relation between these homological dimensions with respect to different groups. It is shown that all homological dimensions of a metric compactum $X$ with respect to any field coincide provided $X$ is homologically locally connected with respect to the singular homology up to dimension $n=\dim X$ (br., $X$ is $lc^n$). We also prove that any two-dimensional $lc^2$ metric compactum $X$ satisfies the equality $\dim (X\times Y)=\dim X+\dim Y$ for any metric compactum $Y$ (i.e., $X$ is dimensionally full-valued). This improves the well known result of  Kodama \cite[Theorem 8]{ko1} that every two-dimensional $ANR$ is dimensionally full-valued. Actually, the condition $X$ to be $lc^2$ can be weaken to the existence at every point $x\in X$ of a neighborhood $V$ of $x$ such that the inclusion homomorphism $H_k(\overline V;\mathbb S^1)\to H_k(X;\mathbb S^1)$ is trivial for all $k=1,2$.
\end{abstract}
\maketitle\markboth{}{Homological dimension}





\section{Introduction}
Everywhere below by a space we mean a metric compactum and by a group an abelian group.

The homological dimension $d_GX$ of a space $X$ with respect to a given group $G$ was introduced by Alexandroff \cite{a} in terms of Vietoris homology: $d_GX$ of a space $X$ is the largest integer $n$ such that there exists a closed set $\Phi\subset X$ carrying an
$(n-1)$-dimensional cycle on $\Phi$ which is not-homologous to zero in $\Phi$, but is homologous to zero in $X$. According to Lefschetz
\cite[Theorem 26.1]{le}, the Vietoris homology is isomorphic to \v{C}ech homology $H_*(X;G)$, where $G$ is considered as a discrete group. So, $d_GX$ can be defined as the largest integer $n$ such that there exists a closed set $\Phi\subset X$ and a non-trivial element $\gamma\in H_{n-1}(\Phi;G)$ with
 $i^{n-1}_{\Phi,X}(\gamma)=0$, where $i^{n-1}_{\Phi,X}:H_{n-1}(\Phi;G)\to H_{n-1}(X;G)$ is the homomorphism generated by the inclusion $\Phi\hookrightarrow X$.

Because \v{C}ech homology is not exact, some authors prefer to define homological dimension involving exact homology groups. In particular, Sklyarenko and his students are using the following definition (see \cite{sk}, \cite{sk1}, \cite{ha}): the homological dimension of $X$, denoted by $h\dim_GX$, is the largest integer $n$ such that there exists a closed subset $\Phi\subset X$ and a non-trivial element of $\widehat{H}_n(X,\Phi;G)$. Here $\widehat{H}_*$ is the exact homology introduced by Sklyarenko in \cite{sk}, which is isomorphic to the Steenrod homology \cite{st} for metric compacta.

 We always have $d_GX\leq\dim X$, and if $0<\dim X<\infty$, then $d_{\mathbb Q_1}X=d_{\mathbb S^1}X=\dim X$, where $\mathbb S^1$ is the circle group and $\mathbb Q_1$ is the group of rational elements of $\mathbb S^1$, see \cite{a}.

 In Section 2 we investigate the relations between
 $d_G$ and $h\dim_G$. For example, we show that if $\dim X=n$ and $X$ is homologically locally connected with respect to the singular homology up to dimension $n$ (br., $X$ is $lc^n$), then $d_G=h\dim_GX$ for any field $G$ (see Corollary 2.3). We apply our results from Section 2 to establish in Section 3 that every two-dimensional compactum $X$ is dimensionally full-valued if $X$ is $lc^2$ (Corollary 3.3). The last result improves a theorem of Kodama \cite[Theorem 8]{ko1}. Actually, the condition in Corollary 3.3 $X$ to be $lc^2$ can be weaken (see Theorem 3.2) to the existence at every point $x\in X$ of a neighborhood $V$ of $x$ such that the inclusion homomorphism $H_k(\overline V;\mathbb S^1)\to H_k(X;\mathbb S^1)$ is trivial for all $k=1,2$.


\section{Some relations between $d_G$ and $h\dim_G$}

It was noted that, in the class of metric compacta, the exact homology $\widehat{H}_*$ is isomorphic to Steenrod homology,  where $G$ is any module over a commutative ring with unity. Moreover, for every module $G$
and a compact pair $(X,A)$ there exists a natural transformation $T_{X,A}:\widehat{H}_{*}(X,A;G)\to H_{*}(X,A;G)$ between the exact and \v{C}ech homologies such that $T^n_{X,A}:\widehat{H}_{n}(X,A;G)\to H_{n}(X,A;G)$ is a surjective homomorphism for each $n$, see \cite{sk} (if $A$ is the empty set, we denote $T^k_{X,\varnothing}$  by
$T^n_{X}$). By \cite[Theorem 4]{sk}, this homomorphism is an isomorphism in the following situations: (i) $\dim X=n$: (ii) $G$ admits a compact topology or $G$ is a vector space over a field; (iii) both the \v{C}ech cohomology group $H^n(X,A;\mathbb Z)$  and $G$ are finitely generated modules having finite numbers of relations.

Let $hd_GX$ be the largest integer $n$ such that there exists a closed set $\Phi\subset X$ and a non-trivial $\gamma\in\widehat{H}_{n-1}(\Phi;G)$ with
$\widehat{i}^{n-1}_{\Phi,X}(\gamma)=0$, where $\widehat{i}^{n-1}_{\Phi,X}:\widehat{H}_{n-1}(\Phi;G)\to\widehat{H}_{n-1}(X;G)$ is the inclusion homomorphism. Clearly, $hd_GX$ is the exact homology analogue of $d_GX$.

\begin{pro}
For any group $G$ we have the following inequalities $d_GX\leq h\dim_GX\leq\dim X$ and $hd_GX\leq h\dim_GX$.
\end{pro}
\begin{proof}
Suppose $d_GX=n$. Then $n\geq 1$ and there is a closed set $\Phi\subset X$ and a non-trivial $\gamma\in H_{n-1}(\Phi;G)$ such that $i^{n-1}_{\Phi,X}(\gamma)=0$. This implies $H_n(X,\Phi;G)\neq 0$, see \cite[Proposition 4.6]{vv}. Consequently, $\widehat{H}_{n}(X,\Phi;G)\neq 0$ (recall that $T^n_{X,\Phi}:\widehat{H}_{n}(X,\Phi;G)\to H_{n}(X,\Phi;G)$ is surjective). Therefore, $d_GX\leq h\dim_GX$. The inequality  $h\dim_GX\leq\dim X$ follows from the fact that all groups $\widehat{H}_{k}(X,\Phi;G)$ are trivial for $k>\dim X$.

For every closed $\Phi\subset X$ and every $n\geq 1$ we have the exact sequence
$$\to\widehat{H}_{n}(X,\Phi;G)\to\widehat{H}_{n-1}(\Phi;G)\to\widehat{H}_{n-1}(X;G)\to\ldots,$$
which yields the inequality $hd_GX\leq h\dim_GX$.
\end{proof}

Recall that a space $X$ is {\em homologically locally connected in dimension $n$} (br., $n-lc$) if for every $x\in X$ and a neighborhood $U$ of $x$ in $X$ there exists a neighborhood $V\subset U$ of $x$ such that the homomorphism $\widetilde{i}_{V,U}^n:\widetilde{H}_n(V;\mathbb Z)\to\widetilde{H}_n(U;\mathbb Z)$ is trivial, where $\widetilde{H}_*(.;.)$ denotes the singular homology groups. The above definition has two variations: (i) if the group $Z$ is replaced by a group $G$, we say that {\em $X$ is $n-lc$ with respect to $G$}; (ii) if every $x\in X$ has a neighborhood $V$ such that the homomorphism  $\widetilde{i}_{V,X}^n:\widetilde{H}_n(V;\mathbb Z)\to\widetilde{H}_n(X;\mathbb Z)$ is trivial, we say that $X$ is semi-$n-lc$. Using the Universal Coefficient Theorem for singular homology, one can show that $X\in n-lc$ with respect to any group $G$
provided $X\in k-lc$ for every $k\in\{n,n-1\}$.
 We say that $X$ is {\em homologically locally connected up to dimension $n$} (br., $X\in lc^n$) provided $X$ is $k-lc$ for all $k\leq n$.

According to \cite[Theorem 1]{mar} the following is true: If $(X,A)$ is a pair of paracompact spaces with both $X$ and $A$ being $lc^n$ and semi-$(n+1)-lc$, then there exists a natural transformation $M_{X,A}$ between the singular and the
\v{C}ech homologies of $(X,A)$ such that for each group $G$ the homomorphisms $M^k_{X,A}:\widetilde{H}_k(X,A;G)\to H_k(X,A;G)$, $k\leq n+1$, are isomorphisms. 
There exists also a similar connection between the singular homology and the exact homology, see \cite[Proposition 9]{sk}: Let $(X,A)$ be a pair of locally compact metric spaces with both $X$ and $A$ being $lc^n$. Then there is a natural transformation $S_{X,A}$ between the singular and the exact homologies with compact supports $\widehat{H}_k^c(X,A;G)$ such that $S^k_{X,A}:\widetilde{H}_k(X,A;G)\to\widehat{H}_k^c(X,A;G)$ is an isomorphism for each $k\leq n-1$ and it is surjective for $k=n$.

Therefore, combining the above results we obtain that if $(X,A)$ is a compact metric pair such that both $X$ and $A$ are $lc^n$, then the homology groups $\widetilde{H}_k(X,A;G)$, $\widehat{H}_k(X,A;G)$ and $H_k(X,A;G)$ are naturally isomorphic for each $k\leq n-1$ and each $G$.

\begin{pro}
If $G$ is any group  and $X$ is $lc^{n}$ with $n=h\dim_GX$, then $d_GX\leq hd_GX=h\dim_GX$.
\end{pro}

\begin{proof}
According to Proposition 2.1, all we need to show is the equality $hd_GX=n$. This is true if $n=1$ because always $1\leq hd_GX\leq h\dim_GX$. So, let $n\geq 2$. By  \cite[Corollary 2]{sk1}, there is a point $x\in X$ such that the module
$H^x_n=\varinjlim_{x\in U} \widehat{H}_n(X,X\setminus U;G)$ is non-trivial and $n$ is the maximal integer with this property. Therefore,  $\widehat{H}_n(X,X\setminus U;G)\neq 0$ for all sufficiently small neighborhoods $U$ of $x$. Because $X$ is $lc^{n}$, it is $p-lc$ with respect to the group $G$ for all $p\in\{n-1,n\}$. Hence, there exists an open neighborhood $V$ of $x$ such that the inclusion homomorphism $\widetilde{i}^{p}_{V,X}:\widetilde{H}_{p}(V;G)\to\widetilde{H}_{p}(X;G)$ is trivial for $p=n$ and $p=n-1$ . According to the mentioned above natural transformation between the singular and the exact homology with compact supports, we have the following commutative diagrams
{ $$
\begin{CD}
\widetilde{H}_{p}(V;G)@>{{\widetilde{i}^{p}_{V,X}}}>>\widetilde{H}_{p}(X;G)\\
@VV{S^{p}_V}V@VV{S^{p}_{X}}V\\
\widehat{H}_{p}^c(V;G)@>{{\widehat{i}^{p}_{V,X}}}>>\widehat{H}_{p}(X;G)
\end{CD}
$$}\\
such that both homomorphisms $S^{p}_V$ and $S^{p}_{X}$ are isomorphisms for $p=n-1$ and surjective  for $p=n$. This implies that the homomorphism
$\widehat{i}^{p}_{V,X}$ is trivial for all $p\in\{n-1,n\}$. Consequently, if $F\subset V$ is any compact set and $p\in\{n-1,n\}$, then the homomorphism
$\widehat{i}^{p}_{F,X}:\widehat{H}_p(F;G)\to\widehat{H}_p(X;G)$ is also trivial.

Now we choose a neighborhood $W$ of $x$ with $\overline W\subset V$ and $\widehat{H}_n(X,X\setminus W;G)\neq 0$.
Then, by the excision axiom,  $\widehat{H}_n(X,X\setminus W;G)$ is isomorphic to $\widehat{H}_n(\overline W,\bd\,\overline W;G)$.
Consider the exact sequence
$$
\xymatrix{
\to\widehat{H}_{n}(\overline W;G)\to\widehat{H}_{n}(\overline W,\bd\,\overline W;G)\to\widehat{H}_{n-1}(\bd\,\overline W;G)\to\ldots}
$$

We claim that $L=\partial\big(\widehat{H}_{n}(\overline W,\bd\,\overline W;G)\big)\neq 0$,
where $\partial$ denotes the boundary homomorphism
$\partial:\widehat{H}_{n}(\overline W,\bd\,\overline W;G)\to\widehat{H}_{n-1}(\bd\,\overline U;G)$. Indeed, otherwise the exactness of the above sequence  yield $\widehat{H}_{n}(\overline W;G)\neq 0$. This means that
$hd_GX\geq n+1$  (recall that, according to the choice of $V$, the homomorphism $\widehat{i}^{n}_{\overline W,X}:\widehat{H}_{n}(\overline W;G)\to\widehat{H}_{n}(X;G)$ should be trivial), a contradiction.

Therefore, $L$ is a non-trivial subgroup of $\widehat{H}_{n-1}(\bd\,\overline W;G)$. Finally, the exactness of the above sequence implies
$\widehat{i}^{n-1}_{\bd\,\overline W,\overline W}(L)=0$. Hence, $\widehat{i}^{n-1}_{\bd\,\overline W,X}(L)$ is also trivial, which yields $hd_GX=n$.
\end{proof}

\begin{cor}
Let $X$ be $lc^{n}$ with $n=h\dim_GX$. Then $d_GX=hd_GX=h\dim_GX$ provided $G$ has the following property:  the homomorphisms
$T^{n-1}_{\bd\,\overline U}:\widehat{H}_{n-1}(\bd\,\overline U;G)\to H_{n-1}(\bd\,\overline U;G)$ are isomorphisms for all open sets $U\subset X$ $($in particular, this is true for any field $G$$)$.
\end{cor}

\begin{proof}
We need to show the equality $d_GX=hd_GX$. It follows from the proof of Proposition 2.2 that there exists a point $x\in X$ and its neighborhood $U$ such that $\widehat{H}_{n-1}(\bd\,\overline U;G)\neq 0$ and the homomorphism $\widehat{i}^{n-1}_{\bd\,\overline U,X}$ is trivial. Obviously, we have the commutative diagram
{ $$
\begin{CD}
\widehat{H}_{n-1}(\bd\,\overline U;G)@>{{\widehat{i}^{n-1}_{\bd\,\overline U,X}}}>>\widehat{H}_{n-1}(X;G)\\
@VV{T^{n-1}_{\bd\,\overline U}}V@VV{T^{n-1}_{X}}V\\
H_{n-1}(\bd\,\overline U;G)@>{{i^{n-1}_{\bd\,\overline U,X}}}>>H_{n-1}(X;G)
\end{CD}
$$}\\
such that
$T^{n-1}_{\bd\,\overline U}$ is an isomorphism.  Therefore,
$H_{n-1}(\bd\,\overline U;G)\neq 0$ and $i^{n-1}_{\bd\,\overline U,X}\big(H_{n-1}(\bd\,\overline U;G)\big)=0$. Hence, $d_GX=n$.
\end{proof}

Recall that the cohomological dimension $\dim_GX$ is the largest integer $m$ such that there exists a closed set $A\subset X$ such that \v{C}ech cohomology group $H^m(X,A;G)$ is non-trivial. It is well known that $\dim_GX\leq n$ iff every map $f\colon A\to K(G,n)$ can be extended to a map $\tilde f\colon X\to K(G,n)$, where $K(G,n)$ is the Eilenberg-MacLane space of type $(G,n)$, see \cite{sp}. We also say that a finite-dimensional metric compactum $X$ is {\em dimensionally full-valued} if $\dim X\times Y=\dim X+\dim Y$ for all metric compacta $Y$, or equivalently (see \cite[Theorem 11]{ku}), $\dim_GX=\dim_{\mathbb Z}X$ for any abelian group $G$.
\begin{cor}
If $X$ dimensionally full-valued and $X$ is $lc^{n}$ with $n=\dim X$, then $d_GX=hd_GX=h\dim_GX=\dim_GX=\dim X$ for any field G.
\end{cor}
\begin{proof}
By \cite{ha}, $h\dim_GX=\dim_GX$. On the other hand, $\dim_GX=\dim X$ because $X$ is dimensionally full-valued. Then, Corollary 2.3 completes the proof.
\end{proof}

One of the main problems concerning homogeneous finite-dimensional $ANR$ compacta is whether any such space is dimensionally full-valued. According to
\cite[Theorem 1.1]{vv1}, any homogeneous $ANR$ compactum satisfies the hypotheses of next proposition. So, this proposition
provides some information about such spaces which are not dimensionally full-valued.

\begin{pro}
Let $X$ be $lc^{n-1}$ with $\dim X=n$ such that each $x\in X$ has a local base $\mathcal B_x$ with the following property: $H^{n-1}(\bd\,\overline U;\mathbb Z)$ is finitely generated for each $U\in\mathcal B_x$. Then $d_{\mathbb Z}X=h\dim_{\mathbb Z}X=n-1$ provided $X$ is not dimensionally full-valued.
\end{pro}

\begin{proof}
Since $n-1\leq h\dim_{\mathbb Z}X\leq n$ and $X$ is not dimensionally full valued, $h\dim_{\mathbb Z}X=n-1$ (see \cite[p.364]{ha}). So, by Proposition 2.2 and the arguments from the proof of Proposition 2.2, we have $d_{\mathbb Z}X\leq n-1$ and there is a point $x\in X$ and its neighborhood $U$ such that $\widehat{H}_{n-2}(\bd\,\overline U;\mathbb Z)\neq 0$ and the homomorphism $\widehat{i}^{n-2}_{\bd\,\overline U,X}$ is trivial.
Since $H^{n-1}(\bd\,\overline U;\mathbb Z)$ are finitely generated, Theorem 4(4) from \cite{sk} implies that
$T^{n-2}_{\bd\,\overline U}:\widehat{H}_{n-2}(\bd\,\overline U;\mathbb Z)\to H_{n-2}(\bd\,\overline U;\mathbb Z)$ is an isomorphism. Hence,
$H_{n-2}(\bd\,\overline U;\mathbb Z)\neq 0$ and triviality of $\widehat{i}^{n-2}_{\bd\,\overline U,X}$ yields triviality of
$i^{n-2}_{\bd\,\overline U,X}$. Therefore, $d_{\mathbb Z}X=n-1$.
\end{proof}

The last statement in this section is the following analogue of Corollary 2.8 from \cite{vv}.
\begin{pro}
Suppose $X$ is a compact space and $d_GX=n$. Then there exists a point $x\in X$ and a local base $\mathcal B_x$ at $x$ such that $\bd U$ contains a closed set $C_U\subset X$ with $H_{n-1}(C_U;G)\neq 0$ for any $U\in\mathcal B_x$.
\end{pro}

\begin{proof}
Since $d_GX=n$, there is a closed set $\Phi\subset X$ such that $i^{n-1}_{\Phi,X}(\gamma)=0$ for some non-trivial $\gamma\in H_{n-1}(\Phi;G)$. By \cite[Lemma 6]{vv}, there exists a closed set $K\subset X$ containing $\Phi$ such that $i^{n-1}_{\Phi,K}(\gamma)=0$ but $i^{n-1}_{\Phi,F}(\gamma)\neq 0$ for all proper closed subsets $F$ of $K$ containing $\Phi$. Choose $x\in K\setminus\Phi$ and let $\mathcal B_x$ consist of all open neighborhoods $U$ of $x$ with $\overline U\cap \Phi=\varnothing$. Denote  $A=K\setminus U$, $B=\overline U\cap K$ and $C_U=\bd_KB$. Then $K=A\cup B$ and $A\cap B=C_U$. Since $A$ is a proper subset of $K$ and contains $\Phi$, $\gamma_U=i^{n-1}_{\Phi,A}(\gamma)$ is a non-trivial element of $H_{n-1}(A;G)$. Consider the following, in general not exact, sequence
{ $$
\begin{CD}
H_{n-1}(C_U;G)@>{{\psi}}>>H_{n-1}(A;G)\oplus H_{n-1}(B;G)@>{{\phi}}>>H_{n-1}(K;G),
\end{CD}
$$}
where $\psi(\theta)=\big(i_{C_U,A}^{n-1}(\theta),i_{C_U,B}^{n-1}(\theta)\big)$ and $\phi\big((\theta_1,\theta_2)\big)=i_{A,K}^{n-1}(\theta_1)- i_{B,K}^{n-1}(\theta_2)$. Since $i_{A,K}^{n-1}(\gamma_U)=i_{\Phi,K}^{n-1}(\gamma)=0$, $\phi\big((\gamma_U,0)\big)=0$.

Take a family $\{\omega_\alpha\}$ of finite open covers of $K$ such that the homology groups $H_{n-1}(C_U;G)$,
$H_{n-1}(A;G)$, $H_{n-1}(B;G)$ and $H_{n-1}(K;G)$ are limit of the inverse systems $\{H_{n}(N_\alpha^{C_U};G),\pi^*_{{\beta},\alpha}\}$,
$\{H_{n-1}(N_\alpha^{A};G),\pi^*_{{\beta},\alpha}\}$, $\{H_{n-1}(N_\alpha^{B};G),\pi^*_{{\beta},\alpha}\}$ and $\{H_{n-1}(N_\alpha^{K};G),\pi^*_{{\beta},\alpha}\}$. Here, $N_\alpha^{F}$, $F\subset K$, denotes the nerve of $\omega_\alpha$ restricted on $F$, and
$\pi_{{\beta},\alpha}:N_{\beta}^F\to N_\alpha^F$ are the corresponding simplicial maps with $\beta$ being a refinement of $\alpha$. Because $\gamma_U\neq 0$, there is $\alpha_0$ with $\gamma_{\alpha_0}=\pi_{\alpha_0}^*(\gamma_U)\neq 0$, where $\pi_\alpha:A\to N_\alpha^{A}$ is the natural map. Since $K$ is a metric space, we can suppose that each cover $\omega_\alpha$ has the following property: if $\{V_j:j=1,..,k\}\subset\omega_\alpha$ such that $\bigcap_{j=1}^{j=k}V_j$ meets both $A$ and $B$, then it also meet $C_U$ (this can be done by considering first a finite open family $\omega'$ in $K$, which covers $C_U$ and satisfies the following condition: for any $\Omega\subset\omega'$ we have $\cap\Omega\neq\varnothing$ if and only if $\cap\Omega$ meets $C_U$; then add to $\omega'$ open subsets of $K$ disjoint from $C_U$ to obtain a cover of $K$). The advantage of this type of covers is that the intersection of the nerves $N_\alpha^{A}$ and $N_\alpha^{B}$ (considered as  sub-complexes of $N_\alpha^K$) is the nerve $N_\alpha^{C_U}$ and $N_\alpha^{A}\cup N_\alpha^{B}=N_\alpha^K$ for all $\alpha$.

Then, we have the Mayer-Vietoris exact sequence (the coefficient group $G$ is suppressed)
{ $$
\begin{CD}
H_{n-1}(N_{\alpha_0}^{C_U})@>{{\psi_{\alpha_0}}}>>H_{n-1}(N_{\alpha_0}^{A})\oplus H_{n-1}(N_{\alpha_0}^{B})@>{{\phi_{\alpha_0}}}>>H_{n-1}(N_{\alpha_0}^{K})\cdots
\end{CD}
$$}
Obviously, $\phi\big((\gamma_U,0)\big)=0$ implies $\phi_{\alpha_0}\big((\gamma_{\alpha_0},0)\big)=0$, and the exactness of this sequence yields
$H_{n-1}(N_{\alpha_0}^{C_U};G)\neq 0$. Therefore, $H_{n-1}(C_U;G)\neq 0$.
\end{proof}

\section{Dimensionally full-valued compacta}

In this section we are going to improve the result of Kodama \cite[Theorem 8]{ko1} that every two-dimensional $ANR$-compactum is dimensionally full-valued.
We need the following statement:
\begin{lem}
Suppose $G$ is a torsion free group and $X$ is compact such that $h\dim_GX=\dim X=n$. If the groups $\widehat{H}_n(X,F;G)$ and $H_n(X,F)$ are isomorphic for each closed set $F\subset X$, then $X$ is dimensionally full-valued.
\end{lem}
\begin{proof}
Since $h\dim_GX=n$, there exists a closed set $\Phi\subset X$ with $\widehat{H}_n(X,\Phi;G)\neq 0$. So, $H_n(X,\Phi;G)\neq 0$, and according to \cite[Proposition 4.5]{vv}, $H_n(X,\Phi;\mathbb Z)$ is also non-trivial. Finally, by \cite[Corollary 1]{ko}, $X$ is dimensionally full-valued.

\end{proof}
\begin{thm}
Suppose $X$ is a two-dimensional compactum satisfying the following condition: for any $x\in X$ there exists a neighborhood $V$ of $x$ such that the homomorphism $i^k_{\overline V,X}:H_k(\overline V;\mathbb S^1)\to H_k(X;\mathbb S^1)$ is trivial for $k=1,2$. Then $X$ is dimensionally full-valued.
\end{thm}
\begin{proof}
Because $\mathbb S^1$ is a compact group, the exact homology $\widehat H_*(.;\mathbb S^1)$ is naturally isomorphic to \v{C}ech homology $H_*(.;\mathbb S^1)$. So, $d_{\mathbb S^1}X=hd_{\mathbb S^1}X$. Moreover, according to \cite{a}, $d_{\mathbb S^1}X=h\dim_{\mathbb S^1}X=2$. The last equality implies the existence of a point $x\in X$ such that $H^x_2=\varinjlim_{x\in U} \widehat{H}_2(X,X\setminus U;\mathbb S^1)$ is non-trivial, see \cite[Corollary 2]{sk1}. Thus,   $\widehat{H}_2(\overline U,\bd\,\overline U;\mathbb S^1)$, being isomorphic to $\widehat{H}_2(X,X\setminus U;\mathbb S^1)$, is not trivial for all sufficiently small neighborhoods $U$ of $x$. On the other hand, there is a neighborhood $V$ of $x$ such that the homomorphisms
$\widehat{i}^k_{\overline V,X}:\widehat{H}_k(\overline V;\mathbb S^1)\to\widehat{H}_k(X;\mathbb S^1)$, $k=1,2$, are trivial. Consequently, the homomorphisms $\widehat{i}^k_{\overline U,X}$, $k=1,2$, are also trivial for all neighborhoods $U$ of $x$ with $\overline U\subset V$. Hence, $\widehat{H}_2(\overline U;\mathbb S^1)=0$ provided $\overline U\subset V$ (otherwise we would have $hd_{\mathbb S^1}X>2$).

Therefore, for any $U$ with $\overline U\subset V$ we have the exact sequence
{ $$
\begin{CD}
0\to\widehat{H}_{2}(\overline U,\bd\,\overline U;\mathbb S^1)@>{{\partial}}>>\widehat{H}_{1}(\bd\,\overline U;\mathbb S^1)\to \widehat{H}_{1}(\overline U;\mathbb S^1)\ldots
\end{CD}
$$}\\
Since $\widehat{H}_{2}(\overline U,\bd\,\overline U;\mathbb S^1)\neq 0$, $\widehat{H}_{1}(\bd\,\overline U;\mathbb S^1)\neq 0$.

So, for all small neighborhoods $U$ of $x$
the groups $\widehat{H}_2(\overline U,\bd\,\overline U;\mathbb S^1)$ and $\widehat{H}_{1}(\bd\,\overline U;\mathbb S^1)$ are non-trivial, while the homomorphisms
$\widehat{i}^{1}_{\overline U,X}:\widehat{H}_{1}(\overline U;\mathbb S^1)\to\widehat{H}_{1}(X;\mathbb S^1)$ 
are trivial. This implies that the homomorphisms 
$\widehat{i}^{1}_{\bd\,\overline U,X}:\widehat{H}_{1}(\bd\,\overline U;\mathbb S^1)\to\widehat{H}_{1}(X;\mathbb S^1)$ are also trivial.
Since $\dim X=2$, $H^{3}(X;\mathbb Z)=H^{3}(\overline U,\bd\,\overline U;\mathbb Z)=H^{3}(\overline U;\mathbb Z)=0$. So, by the Universal Coefficient Theorem (see \cite[Theorem 3]{sk}), we have the isomorphisms
$$\widehat{H}_{2}(\overline U,\bd\,\overline U;\mathbb S^1)\cong\Hom(H^2(\overline U,\bd\,\overline U;\mathbb Z),\mathbb S^1),$$
$$\widehat{H}_{2}(\overline U;\mathbb S^1)\cong\Hom(H^2(\overline U;\mathbb Z),\mathbb S^1)\hbox{~}\mbox{and}\hbox{~}
\widehat{H}_{2}(X;\mathbb S^1)\cong\Hom(H^2(X;\mathbb Z),\mathbb S^1).$$
Similarly, $\dim\bd\,\overline U\leq 1$ yields $H^{2}(\bd\,\overline U;\mathbb Z)=0$. So,
 $$\widehat{H}_{1}(\bd\,\overline U;\mathbb S^1)\cong\Hom(H^1(\bd\,\overline U;\mathbb Z),\mathbb S^1).$$
Thus, $H^2(\overline U,\bd\,\overline U;\mathbb Z)\neq 0$, $H^1(\bd\,\overline U;\mathbb Z)\neq 0$ and the triviality of the homomorphisms
 $\widehat{i}^{1}_{\bd\,\overline U,X}$ yields the triviality of the inclusion homomorphisms
 $j^1_{X\bd\,\overline U}:H^1(X;\mathbb Z)\to H^1(\bd\,\overline U;\mathbb Z)$. On the other hand, it is well known that the simplicial one-dimensional cohomology groups with integer coefficients are free, so any non-trivial one-dimensional \v{C}ech cohomology group $H^1(.;\mathbb Z)$ is torsion free, see for example \cite[Theorem 2.5]{dr}. In particular, $H^1(\bd\,\overline U;\mathbb Z)$ is torsion free.
 It follows from the exact sequence
{ $$
\begin{CD}
\cdots\to H^{1}(X;\mathbb Z)@>{{j^1_{X\bd\,\overline U}}}>>H^1(\bd\,\overline U;\mathbb Z)@>{{\partial_X}}>>H^2(X,\bd\,\overline U;\mathbb Z)\to\cdots
\end{CD}
$$}\\
that $\partial_X$ is an injective homomorphism. Hence, $H^2(X,\bd\,\overline U;\mathbb Z)$ contains elements of infinite order. 
  This implies
$H^2(X,\bd\,\overline U;\mathbb Q)\neq 0$. So, $\dim_{\mathbb Q}X=2$. On the other hand, by \cite[p.364]{ha}, $\dim_{\mathbb Q}X=h\dim_{\mathbb Q}X$. Finally, Lemma 3.1 yields $X$ is dimensionally full-valued.
\end{proof}

\begin{cor}
Every two-dimensional $lc^2$-compactum is dimensionally full-valued.
\end{cor}

\begin{proof}
We already observed the $lc^2$-property implies that the following condition for any group $G$ and open sets $V\subset X$:
\begin{itemize}
\item the groups $\widetilde{H}_k(V,G)$ and $\widetilde{H}_k(X,G)$ are isomorphic
to $H_k(V,G)$ and $H_k(X,G)$, respectively, for all $k\leq 2$ (see \cite[Theorem 1]{mar});
\end{itemize}
Because $X$ is $lc^2$, any point $x\in X$ has a neighborhood $V$ such that the inclusion homomorphisms
$\widetilde{i}^k_{V,X}:\widetilde{H}_k(V;G)\to\widetilde{H}_k(X;G)$ are trivial for $k=1,2$.
So, we can apply Theorem 3.2.
\end{proof}

\textbf{Acknowledgments.} The author thanks the referee for his/her careful reading of the paper.

\end{document}